\documentclass[letterpaper, 10 pt, conference]{ieeeconf}  

\IEEEoverridecommandlockouts                              
\usepackage{cite}
\usepackage{amsmath,amssymb,amsfonts}
\usepackage{algorithmic}
\usepackage{graphicx}
\usepackage{textcomp}
\usepackage{xcolor}
\def\BibTeX{{\rm B\kern-.05em{\sc i\kern-.025em b}\kern-.08em
    T\kern-.1667em\lower.7ex\hbox{E}\kern-.125emX}}
\newtheorem{thm}{Theorem }
\newtheorem{prop}{Proposition}
\newtheorem{prob}{Problem}
\newtheorem{lem}{Lemma }
\newtheorem{defn}{Definition}
\newtheorem{rem}{Remark }
\newtheorem{ass}{Assumption}
\newtheorem{ex}{Example}[section]

\usepackage[ruled,vlined]{algorithm2e}
\usepackage{makecell}

\newcommand{\Real}{\mathbb R}
\newcommand{\eps}{\varepsilon}
\newcommand{\abs}[1]{\left\vert#1\right\vert}
\newcommand{\norm}[1]{\left\Vert#1\right\Vert}
\newcommand{\ra}{\rightarrow}
\newcommand{\set}[1]{\left\{#1\right\}}

\newcommand{\A}{\mathcal{A}}
\newcommand{\D}{\mathcal{D}}

\newcommand{\W}{\mathcal{W}}
\let\subset\subseteq

\title{\LARGE \bf
Towards Learning and Verifying Maximal Neural Lyapunov Functions}

\author{Jun Liu, Yiming Meng, Maxwell Fitzsimmons, and Ruikun Zhou 
\thanks{This research was supported in part by an NSERC Discover Grant, an Ontario Early Researcher Award, and the Canada Research Chairs program. This research was enabled in part by support provided by the Digital Research Alliance of Canada (alliance.ca).}
\thanks{The authors are with the Department of Applied Mathematics, Faculty of Mathematics, 
        University of Waterloo, Waterloo, Ontario N2L 3G1, Canada.  Emails: \texttt{j.liu@uwaterloo.ca, yiming.meng@uwaterloo.ca, mfitzsimmons@uwaterloo.ca, ruikun.zhou@uwaterloo.ca}
        }%
}

\begin{document}

\maketitle
\thispagestyle{empty}
\pagestyle{empty}

\begin{abstract}
The search for Lyapunov functions is a crucial task in the analysis of nonlinear systems. In this paper, we present a physics-informed neural network (PINN) approach to learning a Lyapunov function that is nearly maximal for a given stable set. A Lyapunov function is considered nearly maximal if its sub-level sets can be made arbitrarily close to the boundary of the domain of attraction. We use Zubov's equation to train a maximal Lyapunov function defined on the domain of attraction. Additionally, we propose conditions that can be readily verified by satisfiability modulo theories (SMT) solvers for both local and global stability. We provide theoretical guarantees on the existence of maximal Lyapunov functions and demonstrate the effectiveness of our computational approach through numerical examples.
\end{abstract}
\begin{keywords}
Learning, formal verification, neural networks, nonlinear systems, stability analysis, Zubov's theorem
\end{keywords}

\section{Introduction}

Recent advancements in neural networks and machine learning have revolutionized the landscape of computational research. With the availability of these tools, researchers have been able to develop sophisticated machine learning models that can perform complex tasks with remarkable accuracy. These models are being used in various fields, including image and speech recognition, natural language processing, protein structure prediction and design, and even drug discovery.

Systems and control can potentially leverage the availability of these computational tools to aid the analysis and design of control systems. One of the longstanding challenges in nonlinear control is the construction of Lyapunov functions for stability analysis and controller design. Since Lyapunov's original work \cite{lyapunov1992general} over a hundred years ago, researchers have been searching for constructive approaches to designing Lyapunov functions, and both analytical \cite{haddad2008nonlinear,sepulchre2012constructive} and computational approaches \cite{giesl2007construction, giesl2015review} have been investigated.

In this paper, we propose a neural network approach to constructing Lyapunov functions. In contrast to prior work, we aim to: 1) compute neural Lyapunov functions that can approximate the entire domain of attraction for an asymptotically stable compact set, and 2) formally verify the satisfaction of Lyapunov conditions, both locally and globally. Together, the proposed algorithms provide verified regions of attraction that are close to the domain of attraction, as illustrated by the numerical examples.

\subsection{Related Work}

The computation of Lyapunov functions has a long history, and while it is impossible to mention all related work, we refer readers to a recent survey on computational methods for Lyapunov functions \cite{giesl2015review}, as well as to a more recent survey \cite{dawson2022safe} on the learning of neural Lyapunov functions. 

We would like to highlight a few works that are closely related to ours and have inspired our work. The work in \cite{kang2021data} proposed a data-driven approach for estimating the domain of attraction using Zubov's equation, which is a partial differential equation (PDE) that characterizes the maximal Lyapunov function \cite{vannelli1985maximal} defined on the domain of attraction \cite{zubov1964methods}. However, the approach taken in \cite{kang2021data} does not directly solve Zubov's PDE. Instead, a purely data-driven approach is taken, which relies on simulations of trajectories.

In contrast, the work in \cite{grune2021computing} uses an approach closer to physics-informed neural networks (PINNs) \cite{raissi2019physics} for approximating a solution to Zubov's equation. However, the approach in \cite{grune2021computing} is local in nature, and the Lyapunov conditions used to train the Lyapunov functions are essentially conditions for local exponential stability. Both papers \cite{kang2021data} and \cite{grune2021computing} discuss the potential of using neural networks to break the curse of dimensionality in approximating Lyapunov functions, and our work shares the same optimism \cite{poggio2017and} without discussing this aspect. 

The work in \cite{jones2021converse}, even though focusing on sums of squares (SOS) approaches for approximating Lyapunov functions, is closely related to our work. The partial differential inequality constraint that the authors used to optimize the polynomial Lyapunov functions takes the form of Zubov's PDEs, although not explicitly mentioned as such. 

Neural Lyapunov functions for control are investigated in \cite{chang2019neural}, where the authors use satisfiability modulo theories (SMT) solvers to verify that Lyapunov conditions are met by neural Lyapunov functions. The work was extended to cover systems with unknown dynamics in \cite{zhou2022neural}, while still offering stability guarantees. The search for Lyapunov functions in \cite{chang2019neural, zhou2022neural}, similar to \cite{grune2021computing}, is local in nature. Furthermore, in both \cite{chang2019neural, zhou2022neural}, SMT verification for Lyapunov stability conditions omits a region around the origin. We formally address this issue here as a side result. 

\section{Preliminaries}

\subsection{Set Stability}

Consider an autonomous nonlinear system of the form
\begin{equation}
    \label{eq:sys}
    \dot x  = f(x),
\end{equation}
where $f:\,\Real^n\ra\Real^n$ is a locally Lipschitz function. We are interested in characterizing uniform asymptotic stability of any compact invariant set for system (\ref{eq:sys}) using Lyapunov functions approximated by neural networks. We denote the unique solution to (\ref{eq:sys}) from the initial condition $x(0)=x_0$ by $\phi(t,x_0)$ for $t\in J$, where $J$ is the maximal interval of existence for $\phi$.  

\begin{defn}[Invariant set]
A set $\A\subset\Real^n$ is said to be a \emph{positively invariant set} of (\ref{eq:sys}) if all solutions of (\ref{eq:sys}) starting in $\A$ remain in $\A$ in positive time, i.e., $x_0\in \A$ implies $\phi(t,x_0)\in \A$ for all $t\ge 0$. 
\end{defn}

Since we only consider positively invariant sets in this paper, we will simply refer to them as invariant sets. 
The distance of a point $x\in\Real^n$ to a set $\A\subset \Real^n$ is defined as $\norm{x}_{\A}:=\inf_{y\in \A}\norm{x-y}$. 

\begin{defn}[Set stability]\label{def:stability}
A set $\A\subset\Real^n$ is said to be \textit{uniformly asymptotically stable} (UAS) for (\ref{eq:sys}) if the following two conditions are met:
\begin{enumerate}[(1)]
    \item For every $\eps>0$, there exists a $\delta>0$ such that $\norm{x_0}_\A<\delta$ implies $\norm{\phi(t,x_0)}_\A<\eps$ for all $t\ge 0$; and 
    \item There exists some $\rho>0$ such that, for every $\eps>0$, there exists some $T>0$ such that $\norm{\phi(t,x_0)}_\A<\eps$ whenever $\norm{x_0}_\A<\rho$ and $t\ge T$. 
\end{enumerate}
\end{defn}

Note that if a closed set $\A$ is stable (i.e., satisfying item 1) in Definition \ref{def:stability}), then it has to be positively invariant \cite[Theorem 1.6.6]{bhatia1967stability}. While the definition above is for any set $\A$, it is of more interest to consider stability of closed sets, because stability may be trivially satisfied if the set $\A$ is not closed (see \cite[Remark 1.6.7 and Example 1.6.8]{bhatia1967stability}). 

The definition above requires that there exists a neighborhood of $\A$ that are uniformly attracted to $\A$. To capture the largest region that is attracted to a uniformly asymptotically stable set $\A$, the domain of attraction of $\A$ is defined as follows. 

\begin{defn}[Domain of Attraction]
    Given a set $\A\subset \Real^n$ that is uniformly asymptotically stable for (\ref{eq:sys}), the \textit{domain of attraction} of $\A$ with respect to (\ref{eq:sys}) is defined as
    $$
    \D(\A): = \set{x\in\Real^n:\,\lim_{t\ra \infty}\norm{\phi(t,x)}_{\A} = 0}.   
    $$ 
    Any invariant subset of $\D(\A)$ is called a \textit{region of attraction}. 
\end{defn}

We know that the domain of attraction of any set is an open set \cite[Theorem 1.6.31]{bhatia1967stability}. 

We shall use the following standing assumption of the paper. 

\begin{ass}
We have a compact set $\A$ that is uniformly asymptotically stable for (\ref{eq:sys}). 
\end{ass}

\subsection{Maximal Lyapunov Function and Zubov's Theorem}

The domain of attraction can be characterized by a maximal Lyapunov function \cite{vannelli1985maximal} as described in the following theorem. 

\begin{thm}\label{thm:max_lyap}
	Let $D\subset\Real^n$ be an open set. Suppose that there exists a continuous function $V:\,D\ra \Real$ such that $\A\subset D$ and the following conditions hold:
 \begin{enumerate}
     \item $V$ is positive definite on $D$ with respect to $\A$, i.e., $V(x)=0$ for all $x\in \A$ and $V(x)>0$ for all $x\in D\setminus \A$;
     \item the derivative of $V$ along solutions of (\ref{eq:sys}), given by
     $$
     \dot V(x):= \lim_{t\ra 0+} \frac{V(\phi(t,x))-V(x)}{t}, 
     $$
     is well defined for all $x\in D$ and satisfies 
     \begin{equation}\label{eq:DV1}
     \dot V(x) = - \Phi(x),    
     \end{equation}
     where $\Phi:\,D\ra\Real$ continuous and positive definite with respect to $\A$;
     \item $V(x)\ra \infty$ as $x\ra \partial D$ or $\norm{x}\ra \infty$. Then $D=\D(\A)$. 
 \end{enumerate}
\end{thm}

Under mild assumptions, the existence of a maximal Lyapunov function on $\D(\A)$ is also necessary \cite{vannelli1985maximal}, following a converse Lyapunov argument \cite{massera1956contributions}. Furthermore, if $f$ is Lipschitz continuous on $\D(A)$, then such a $V$ can be chosen to be continuously differentiable\footnote{While the results in \cite{vannelli1985maximal} were proved for case of an asymptotically stable equilibrium point for (\ref{eq:sys}), they can be readily extended to set stability.}. 

Theorem \ref{thm:max_lyap} is closely related, and in fact equivalent, to Zubov's theorem \cite{zubov1964methods} stated below. 

\begin{thm}\label{thm:zubov}
Let $D\subset\Real^n$ be an open set containing $\A$. Then $D=\D(\A)$ if and only if there exists two continuous functions $W:\,D\ra \Real$ and $\Psi:\,D\ra \Real$ such that the following conditions hold:
 \begin{enumerate}
     \item $0<W(x)<1$ for all $x\in D$ and $W(x)=0$ for all $x\in \A$; 
     \item $\Psi$ is positive definite on $D$ with respect to $\A$;  
     \item for any sufficiently small $c_3>0$, there exist $c_1$ and $c_2$ such that $\norm{x}_\A\ge c_3$ implies 
     $W(x)>c_1$ and $\Psi(x)>c_2$;  
     \item $W(x)\ra 1$ as $x\ra y$ for any $y\in \partial D$;
     \item $W$ and $\Psi$ satisfy
     \begin{equation}\label{DV:zubov}
     \dot W (x) = -\Psi(x)(1-W(x)), 
     \end{equation}
     where $\dot W$ is the derivative of $W$ along solutions of (\ref{eq:sys}) as defined in item 2 of Theorem \ref{thm:max_lyap}. 
 \end{enumerate}
\end{thm}

\begin{rem}\label{rem:vw}\em 
Theorem \ref{thm:max_lyap} and Theorem \ref{thm:zubov} can be related by the following equation
\begin{equation}\label{eq:V2W}
    W(x) = 1 - \exp(-\alpha V(x)), 
\end{equation}
for some constant $\alpha>0$.  It is easy to verify that $V$ satisfying (\ref{eq:DV1}) implies 
 \begin{align*}
 \dot W(x) &= \alpha\exp(-\alpha V(x))\dot{V}\\
 &=-\alpha\exp(-\alpha V(x))\Phi(x) = -\alpha(1-W(x)) \Phi(x),
 \end{align*}
 which verifies (\ref{DV:zubov}) with $\Psi=\alpha\Phi(x)$. Conversely, one can take 
 \begin{equation}\label{eq:W2V}
 V(x)=-\ln(1-W(x))/\alpha 
 \end{equation}
 to verify (\ref{eq:DV1}) using (\ref{DV:zubov}).  
 Clearly, the choice of such transformation from $V$ to $W$, and vice versa, is not unique. For example, we can also define 
 \begin{equation}\label{eq:V2W2}
    W(x) = \tanh(\alpha V(x)). 
\end{equation}
Then we also have
 \begin{align*}
 \dot W(x) &= \alpha(1-\tanh^2(\alpha V(x)))\dot{V}\\
 & = -\alpha(1-W(x))(1+W(x))\Phi(x),
 \end{align*}
 which verifies (\ref{DV:zubov}) with $\Psi=\alpha(1+W(x))\Phi(x)$. It is obvious that there are infinitely many such scalar transformations (see Section \ref{sec:constructW}). Hence, Zubov's PDE (\ref{DV:zubov}) is not unique and we need to specify $\Psi$ before solving it.  
\end{rem}

\begin{rem}\em 
One potential advantage of Zubov's construction of a Lyapunov function, namely $W(x)$, is that it is bounded and its value approaches one as $x$ approaches the boundary, whereas the function $V(x)$ in Theorem \ref{thm:max_lyap} has values approaching infinity as $x$ tends to the boundary. A bounded function can offer potential advantages in numerical approximations, especially when we are interested in solving a PDE to find a Lyapunov function, as will be described later in this paper. The fact that $W(x)$ approaches a finite value as $x$ approaches the boundary of its domain also makes it possible to extend its domain to the entire space $\Real^n$ (or any desired set on which computation takes place).
\end{rem}

Let $W:D\ra\mathbb R$ be a function satisfying the conditions in Zubov's theorem, and then $D=\D(\A)$. Suppose that the definition of $W$ is extended to a set $X$ containing $\D(\A)$ with $W(x)=1$ for all $x\in X\setminus \D(\A)$. Clearly, the domain of attraction of $\A$ is completely characterized by the strict sublevel-1 set of $W$, i.e.,
$$
\D(\A) = \set{x\in X:\, W(x)<1}. 
$$

\subsection{Problem Statement}

The objective of this paper is to compute neural network approximations of a maximal Lyapunov function in the form of a function satisfying Zubov's PDE (\ref{DV:zubov}), and use these neural Lyapunov functions to certify regions of attraction for $\A$, which are characterized by sublevel sets of the neural Lyapunov functions. We rely on satisfiability modulo theories (SMT) solvers to verify that the Lyapunov conditions are met. Since verification is done over bounded domains and neural approximations offer guarantees over compact domains, we put forward the following assumption.

\begin{ass}
    We have  $\D(\A)\subset X$ for some compact set $X$. 
\end{ass}

It is well known that the boundary of $\D(\A)$ consists of complete trajectories of (\ref{eq:sys}) \cite{bhatia1967stability}. For instance, in a two-dimensional system, a period orbit can bound a domain of attraction for a stable equilibrium point. 

If the set $\D(\A)$ is unbounded and suppose that $\A$ is contained in the interior of $X$, the proposed algorithms can still certify regions of attractions inside $X$, although losing guarantees of approximating $\D(\A)$. 

The main problem to be addressed is stated next. 

\begin{prob}
Given a uniformly asymptotically stable set $\A$,  compute neural Lyapunov functions that can certify regions of attraction for $
\A$ with provable guarantees. 
\end{prob}

\section{A Constructive Maximal Converse Lyapunov Function}\label{sec:construct}

In this section, we use ideas from converse Lyapunov function theory to explicity construct a maximal Lyapunov-barrier function for a given uniformly asymptotically set $\A$ and establish its properties. 

\subsection{Construction of $V$}

Let $\omega:\,\Real^n\ra \Real$ be a continuous function that is positive definite with respect to $\A$. Define 
\begin{equation}
    \label{eq:V}
V(x)=\int_0^\infty \omega(\phi(t,x))dt,\quad x\in \Real^n, 
\end{equation}
where if the integral diverges, we let $V(x)=\infty$. 

The following is assumed for the function $\omega$ and the resulting function $V$ defined by (\ref{eq:V}).

\begin{ass}\label{as:finite}
The following items hold true:
\begin{enumerate}
    \item For any $\delta>0$, there exists $c>0$ such that $\omega(x)>c$ for all $\norm{x}_\A>\delta$. 
    \item There exists some $\rho>0$ such that the integral $V(x)$ defined by (\ref{eq:V}) converges for all $x$ such that $\norm{x}_\A<\rho$. 
    \item For any $\eps>0$, there exists $\delta>0$ such that $\norm{x}_\A<\delta$ implies $V(x)<\eps$. 
\end{enumerate}
\end{ass}

\begin{rem}\em 
We can show that Assumption \ref{as:finite} holds if the set $\A$ is locally exponentially stable and $\omega$ is locally Lipschitz. A typical choice for $\omega$ is given by $\omega(x)=\norm{x}_\A$. 
\end{rem}

Under this assumption, we can verify that $V$ satisfies the conditions for a maximal Lyapunov function stated in Theorem \ref{thm:max_lyap}. 

\begin{prop}\label{prop:V}
The function $V:\,\Real^n\ra\Real\cup\set{\infty}$ defined by (\ref{eq:V}) satisfies the following: 
\begin{enumerate}
    \item $V(x)<\infty$ if and only if $x\in \D(\A)$;
    \item $V(x)\ra \infty$ as $x\ra \partial \D(\A)$; 
    \item $V$ is positive definite with respect to $\A$; 
    \item $V$ is continuous on $\D(\A)$ and 
    \begin{equation}
        \dot{V}(x) = -\omega(x),
    \end{equation}
    for all $x\in \D(\A)$. 
\end{enumerate}
\end{prop}

\begin{proof}
1) Suppose $x\in \D(\A)$. Let $\rho>0$ be from Assumption \ref{as:finite}. There exists some $T_x>0$ such that $\norm{\phi(t,x)}_{\A}<\rho$ for all $t\ge T_x$. It follows that
\begin{align*}
V(x) & = \int_0^{T_x} \omega(\phi(t,x))dt + \int_{T_x}^\infty  \omega(\phi(t,x))dt    \\
& = \int_0^{T_x} \omega(\phi(t,x))dt + \int_{0}^\infty \omega(\phi(t,\phi(T_x,x)))dt\\
& = \int_0^{T_x} \omega(\phi(t,x))dt + V(\phi(T_x,x)) < \infty,
\end{align*}
where we used finiteness of the first integral and Assumption \ref{as:finite} to conclude $V(\phi(T_x,x))<\infty$. Now suppose that $x\not \in \D(\A)$. Then $\phi(t,x)\not\in \D(\A)$ for any $t\ge 0$. Since $\D(\A)$ is open and $\A$ is compact, there exists some $\delta>0$ such that  $\norm{\phi(t,x)}_\A>\delta$ for all $t\ge 0$. By Assumption \ref{as:finite}, $V(x)=\infty$. 

2) Let $\set{x_n}$ be a sequence such that  $x_n\ra y$ for some $y\in \partial \D(\A)$. Choose $\delta\in (0,\rho)$, where $\rho>0$ is from the definition of set stability of $\A$. Let $T_n$ be the first time that $\norm{\phi(t,x_n)}_\A\le \delta$. Then by continuity, $\norm{\phi(T_n,x_n)}_\A=\delta$ and for all $t\in [0,T_n)$, we  have $\norm{\phi(t,x_n)}_\A> \delta$ and $\omega(\phi(t,x_n))>c$ for some $c>0$ because of Assumption \ref{as:finite}. Hence 
$$
V(x_n) = \int_0^\infty \omega(\phi(t,x_n)) dt \ge \int_0^{T_n} \omega(\phi(t,x_n)) dt \ge cT_n.
$$
We can conclude $V(x_n)\ra \infty$ if $T_n\ra\infty$. Suppose that this is not the case. Then $\set{T_n}$ contains a bounded subsequence, still denoted by $\set{T_n}$, that converges to $T$. It follows that $\phi(T_n,x_n)\ra \phi(T,y)$ and $\norm{\phi(T,y)}_\A=\delta<\rho$. Hence $\phi(T,y)\in \D(\A)$. It follows that $y\in \D(\A)$. Since $\D(\A)$ is open, this is a contradiction. We must have $V(x_n)\ra\infty$ as $n\ra\ \infty$. 

3) Positive definiteness of $V$ follows from positive definiteness of $\omega$ and by a continuity argument. 

4) For $x\in\D(\A)$, we have 
\begin{align*}
\dot{V}(x) & = \lim_{t\ra 0+} \frac{V(\phi(t,x))-V(x)}{t}\\
& =\lim_{t\ra 0+}\frac{\int_0^\infty \omega(\phi(s,\phi(t,x)))ds - \int_0^\infty \omega(\phi(s,x))ds}{t} \\
& =  \lim_{t\ra 0+}\frac{\int_0^\infty \omega(\phi(s+t,x))ds -\int_0^\infty \omega(\phi(s,x))ds}{t}\\
& = \lim_{t\ra 0+}\frac{\int_{t}^\infty \omega(\phi(s,x))ds - \int_0^\infty \omega(\phi(s,x))ds}{t}\\
&= \lim_{t\ra 0+} \frac{\int_{t}^0 \omega(\phi(s,x))ds }{t} = -\omega(x).
\end{align*}
The proof is complete. 
\end{proof}

\subsection{Construction of $W$}\label{sec:constructW}

Let $\beta:\,[0,\infty)\ra\Real$ satisfy 
\begin{equation}\label{eq:beta}
    \dot \beta =  (1-\beta)\psi(\beta)\,\quad \beta(0)=0,
\end{equation}
where $\psi:\,[0,\infty)\ra\Real$ is a locally Lipschitz function satisfying $\psi(s)>0$ for $s>0$.  

\begin{lem}\label{lem:beta}
Any function satisfying (\ref{eq:beta}) is continuously differentiable, strictly increasing, and satisfies $\beta(0)=0$ and  $\beta(s)\ra 1$ as $s\ra \infty$. 
\end{lem}

With $V$ defined by (\ref{eq:V}) and satisfying the conditions in Proposition \ref{prop:V}, let 
\begin{equation}
    \label{eq:W}
W(x) = \left\{\begin{aligned}
&\beta(V(x)),\text{ if } V(x)<\infty,\\
&1, \quad  \text{otherwise},
\end{aligned}\right.
\end{equation}
where $\beta:\,[0,\infty)\ra\Real$ satisfies (\ref{eq:beta}). 

We can verify the following properties for $W$. 

\begin{prop}\label{prop:W}
The function $W:\,\Real^n\ra\Real$ defined by (\ref{eq:W}) satisfies the conditions in Theorem \ref{thm:zubov} on $\D(\A)$.  Furthermore, $W$ is continuous on $\Real^n$.  
\end{prop}

\begin{proof}
Items (1), (2), (4) of Theorem \ref{thm:zubov} directly follow from Proposition \ref{prop:V} and Lemma \ref{lem:beta}. To verify item (3), by Assumption \ref{as:finite}, we know 
for any $\delta>0$, there exists $c>0$ such that $\omega(x)>c$ for all $\norm{x}_\A>\delta$. By local Lipschitz of $f$, it takes at least $\tau>0$ time for any solutions to reach from the level set $\set{\norm{x}_A=2\delta}$ to $\set{\norm{x}_A=\delta}$. Hence for $V(x)\ge c\tau$ for all $x$ such that $\norm{x}_\A>2\delta$. By the properties of $\beta$, item (3) of Theorem \ref{thm:zubov} holds for $W$. Continuity of $W$ follows from its definition outside $\D(\A)$ and item (4) of Theorem \ref{thm:zubov}. 

To check (5), taking the derivative of $W(x)$ along solutions of (\ref{eq:sys}) gives
\begin{align*}
    \dot{W}(x)& = \dot{\beta}\dot{V} = -(1-\beta(V(x)))\psi(\beta(V))\omega(x) \\
    &= -(1-W(x))\Phi(x),
\end{align*}
where $\Phi(x) = \psi(\beta(V))\omega(x)$. 
\end{proof}

\begin{rem}\em 
    The constructions in Remark \ref{rem:vw} are special cases of Proposition \ref{prop:W}. In particular, (\ref{eq:V2W}) is given by setting $\psi(s)=\alpha$ for some constant $\alpha>0$ and (\ref{eq:V2W2}) is given by $\psi(s) = \alpha(1+s)$. 
\end{rem}

\begin{ex}[A concrete example of $W$]
    Consider the scalar system
    $\dot x = -x + x^3$. It has three equilibrium points at $\set{0,\pm 1}$. The equilibrium point $\A = \set{0}$ is uniformly asymptotically stable with $\D(\A)=(-1,1)$. Consider 
    $$
    V(x) = \int_0^\infty \abs{\phi(t,x)}^{2}dt,
    $$
    where $\beta>0$. By Proposition \ref{prop:V}, we have
    $$
    \dot V(x) = V'(x)(-x+x^3) = - x^2.
    $$
    For $x\in (0,1)$, assuming differentiability of $V$, we have
    $$
    V'(x) = \frac{x}{1-x^2}.
    $$
    Integrating this with the condition $V(0)=0$ gives 
    $$
    V(x) = -\frac12 \ln(1-x^2).
    $$
    Taking $W(x)=1-\exp(-\alpha V(x))$ as in (\ref{eq:V2W}), we obtain
    $$
    W(x) = 1 - (1-x^2)^{\frac{\alpha}{2}},\quad \alpha>0. 
    $$
    One can easily verify that $W(x)$ satisfies the conditions in Zubov's theorem \ref{DV:zubov}. 
\end{ex}

\section{Training Lyapunov Functions with Physics-Informed Neural Networks}\label{sec:algorithm}

Given the constructions of $V(x)$ and $W(x)$ in Section \ref{sec:construct}, we propose a physics-informed neural network (PINN) approach to train neural Lyapunov functions that can approximately capture the domain of attraction. Furthermore, the PINN approach can be augmented with a data-driven approach to improve its performance. 

In Algorithm \ref{alg:pinn}, 
a neural network is trained to solve the Zubov's PDE with a particular choice of $\Psi(x)$. For the numerical examples in this paper, we choose 
\begin{equation}\label{eq:Psi}
\Psi(x) = \alpha(1+W(x))\Phi(x),
\end{equation}
with $\Phi(x)=\norm{x}^2_\A$. The details of constructing the loss function are discussed next. 

Let $W_N(x;\theta)$ be a neural approximation for solving Zubov's PDE (\ref{DV:zubov}). The loss function consists of three terms:
\begin{equation}
    \label{eq:loss}
    \mathcal{L}(\theta) = L_r(\theta)  + L_b(\theta) + L_d(\theta), 
\end{equation}
where $L_r$ is the residual error of the PDE, evaluated over a training set $S=\set{x_i}_{i=1}^{N}\subset X$, where $X$ is assumed to be a compact set on which training takes place. For example, we can choose it to be the mean-square error defined by 
\begin{equation}
    \label{eq:L_r}
    L_r = \frac{1}{N}\sum_{i=1}^N(\nabla_{x} W_N(x_i;\theta) f(x_i) + \Psi(x_i)(1-W_N(x_i;\theta)))^2. 
\end{equation}
The loss $L_b$ captures the boundary conditions. In this case, we want $W(x)=1$ for $x\not \in\D(\A)$ and $W(x)=0$ for all $x\in A$. While it may not be able to precisely define the boundaries, some approximation knowledge suffices. For example, if an over-approximation $U$ of $\D(\A)$ is known, the loss can be defined as 
\begin{equation}
    \label{eq:L_b}
    L_b = \frac{1}{N_b}\sum_{x_i\in S\cap (X\setminus U)}(W_N(x_i;\theta)-1)^2,
\end{equation}
where $N_b$ is the number of points in $S\cap X\setminus U$. If $\A$ is an equilibrium point, adding $W(0)=(0)$ alone may not be ideal. If the equilibrium point is exponentially stable, we know that locally $V(x)$ constructed in (\ref{eq:V}) with $\omega(x)=\norm{x}^2$ is a Lyapunov function for exponential stability in the sense that
\begin{equation} \label{eq:c1Vc2}
c_1\norm{x}^2 \le V(x) \le c_2 \norm{x}^2,
\end{equation}
and 
$$
\dot{V} (x) \le -c_3\norm{x}^3,
$$
for $x$ in a small neighborhood of the origin. Hence it is without loss of generality to assume that (\ref{eq:c1Vc2}) holds locally, which translates to a boundary constraint to $W$ as
\begin{equation} \label{eq:c1Wc2}
\beta(c_1\norm{x}^2) \le W(x) \le \beta(c_2 \norm{x}^2),
\end{equation}
where $\beta$ is given by (\ref{eq:beta}). As discussed in Remark \ref{rem:vw}, we can simply choose 
$\beta(s)=(\tanh(\alpha s))$ or $\beta(s)=1-\exp(-\alpha s)$ for some $\alpha>0$, which needs to be consistent with the PDE that is being solved. For the numerical examples in this paper, we choose $\beta(s)=(\tanh(\alpha s))$, which is consistent with the PDE (\ref{DV:zubov}) with $\Psi$ given by (\ref{eq:Psi}). 

\begin{rem}\em
   Inequality constraints can be encoded in the cost function by defining a loss term $L_+= \max(h,0)$ or $L_-=\min(h,0)$. It is easy to see that $L_+=0$ if and only if $h\le 0$, and similarly, $L_-=0$ if and only if $h\ge 0$. It is suggested in \cite{grune2021computing} that relaxing a Zubov-type equation $\dot{V}(x)=-\norm{x}^2$ to an inequality $\dot{V}(x)\le -\norm{x}^2$ can improve the process of training $V$. However, in our examples, we did not notice any issues using physics-informed neural networks to solve the Zubov PDE directly. Solving the PDE more globally, as opposed to locally in the examples in \cite{grune2021computing}, seems to help the training process. 
\end{rem}

Finally, $L_d$ is a data loss defined by an optional set of data points $S_d=\set{y_i}_i^{N_d}$ on which approximate values $\hat W(y_i)$ for $W(y_i)$ can be obtained. In this case, $W(y_i)$ can be approximated by simulating $\phi(t,y_i)$ for a sufficiently long period of time until it converges to the set $\A$, or when the value of $V(y_i)$ defined by the integral in (\ref{eq:V}) exceeds a certain threshold. To be consistent with the PDE that is being solved, we should use (\ref{eq:V2W2}) to evaluate $\hat W(y_i)$ based on simulated data for $V(y_i)$. Based on the values of $\set{\hat W(y_i)}_i^{N_d}$, the data loss $L_d$ can be defined as 
\begin{equation}
    \label{eq:L_d}
    L_d = \frac{1}{N_d}\sum_{i=1}^{N_d}(W_N(y_i;\theta)-\hat W(y_i))^2.
\end{equation}

\begin{algorithm}
\SetAlgoLined
\caption{Physics-Informed Neural Network (PINN) Training for Solving Zubov's PDE (\ref{DV:zubov})}
\label{alg:pinn}
\SetKwInput{Input}{Input}
\SetKwInput{Output}{Output}
\Input{Training data $S=\set{x_i}_{i=1}^{N}$, $\set{(y_i,\hat{W}(y_i)}_{i=1}^{N_d}$, Zubov's PDE (\ref{DV:zubov}), batch size $B$, learning rate $\eta$, maximum iterations $M$}
\Output{Trained neural Lyapunov function $W_N(x;\theta^*)$}
\BlankLine
Initialize neural network parameters $\theta$\;
Define loss function $\mathcal{L}(\theta)$ as in (\ref{eq:loss})\;
\For{$m = 1, 2, \dots, M$}{
    Randomly sample a batch of $B$ training examples: $S_m = \set{x_i}_{i=1}^{B} \subseteq S$\;
    Compute gradients $\nabla_x W_N(x_i;\theta)$ for $i=1,\dots,B$ using automatic differentiation\;
    Compute gradients $\nabla_\theta \mathcal{L}(\theta)$ using automatic differentiation with respect to the batch of samples $S_m$\;
    Update parameters: $\theta \leftarrow \theta - \eta \nabla_\theta \mathcal{L}(\theta)$\;
    Compute the current loss: $\mathcal{L}_m = \mathcal{L}(\theta)$\;
    \If{$\mathcal{L}_m$ converges to below a threshold}{
        \textbf{break}\;
    }
}
\end{algorithm}

\section{Formal Verification of Neural Lyapunov Functions}

One perceived weakness of neural network approximation is its lack of formal guarantees. To overcome this, formal verification of the learned neural Lyapunov functions is conducted using satisfiability modulo theories (SMT) solvers. In this section, we will discuss the details of verifying neural Lyapunov functions using SMT solvers, and for simplicity, we will limit our discussion to the case of a stable equilibrium point at the origin. We discuss verification of both local and global stabiltiy. 

\subsection{Verification of Local Stability and Region of Attraction}
\label{sec:local}

In this section, we assume that $f$ is continuously differentiable. We further assume that the origin is exponentially stable, i.e., (\ref{eq:sys}) admits a linearization 
\begin{equation}
    \label{eq:linear}
    \dot x = Ax,
\end{equation}
where $A$ is a Hurwitz matrix. Now write the nonlinear system (\ref{eq:sys}) as 
\begin{equation}
    \label{eq:nonlinear}
    \dot x = Ax + g(x),
\end{equation}
where $g(x)=f(x)-Ax$ satisfies $\lim_{x\ra 0}\frac{\norm{g(x)}}{\norm{x}}=0$. The linear system (\ref{eq:linear}) has a quadratic Lyapunov function $V(x)=x^TPx$ given by solving $P$ from the Lyapunov equation 
\begin{equation}
    \label{eq:lyap}
    PA +A^T P = -Q,
\end{equation}
for any positive definite matrix $Q$. 
Consider two sets $$
S_{r}:=\set{x\in\Real^n:\,2x^TP g(x)\le  r\norm{x}^2}
$$ 
and 
$$
\Omega_c:=\set{x\in\Real^n:\,x^TPx\le c}.
$$
Suppose that we can choose $c>0$ and $r>0$ such that $\Omega_c\subset S_r$. Then for all $x\in \Omega_c$, we have 
\begin{align*}
\dot{V}(x) & = x^T(PA+A^TP)x + 2x^TPg(x)\\
&= - x^TQx + 2x^TPg(x)\\
& \le (-\lambda_{\min}(Q)+r)\norm{x}^2.
\end{align*}
If $r<\lambda_{\min}(Q)$, then we verified local stability of the origin and obtained $\Omega_c$ as a verified invariant set that is attracted to the origin. To summarize, the conditions required for verifying local stability of the origin and $\Omega_c$ being a local region of attraction are 
\begin{align}
r &< \lambda_{\min}(Q),\label{eq:rQ}\\
x^TPx\le c &\Longrightarrow 2x^TP g(x)\le  r\norm{x}^2.\label{eq:cr}
\end{align}
The inequality (\ref{eq:rQ}) can be readily verified by computing the eigenvalue of $Q$ and picking $r$ accordingly. To verify (\ref{eq:cr}), one could use sums of squares (SOS) or SMT verification. In this paper, we emphasize the use of SMT solvers because it can potentially handle a broader class of nonlinear functions \cite{gao2013dreal}. However, at a glance, inequality (\ref{eq:cr}) may not be amenable to SMT verification because the identity holds trivially at the origin. Hence $\delta$-decidable tools such as dReal \cite{gao2013dreal} may have trouble verifying (\ref{eq:cr}) directly, as the origin always satisfies a $\delta$-weakened version of (\ref{eq:cr}) for any $\delta>0$. Fortunately, a closer examination of the inequality (\ref{eq:cr}) and the assumption on $g$ reveals that, since $
g$ is continuously differentiable, we can leverage this fact to help verify (\ref{eq:cr}). Indeed, by the mean value theorem, we have
\begin{equation}
    P g(x) = P g(x) - P g(0) = \int_0^1 P\cdot Dg(tx)dt\cdot x,
\end{equation}
where 
$Dg$ is the Jacobian of $g$ given by $Dg= Df - A$. Clearly, $Dg(0)=Df(0)-A=0$ by construction. Hence
\begin{equation}
\norm{P g(x)} \le \sup_{0\le t\le 1}\norm{P\cdot  Dg(tx)}\norm{x},
\end{equation}
which implies
$$
 2x^TP g(x) \le 2\sup_{0\le t\le 1}\norm{P \cdot Dg(tx)}\norm{x}^2.
$$
As a result, to verify (\ref{eq:cr}), we just need to verify 
\begin{align}
x^TPx\le c &\Longrightarrow 2\sup_{0\le t\le 1}\norm{P\cdot Dg(tx)}\le r.\label{eq:cr2}
\end{align}
Since $Dg(0)=0$ and $Dg$ is continuous, for any $r>0$, one can always choose $c>0$ sufficiently small such that (\ref{eq:cr2}) can be verified. 

\begin{ex}[Reversed Van der Pol]\label{ex:van_der_pol_local}
Consider the reversed Van der Pol oscillator
\begin{align*}
\dot x_1 &= -x_2, \\
\dot x_2 &= x_1 - (1 - x_1^2)x_2, 
\end{align*}
which has a stable equilibrium point at the origin. The linearization at the origin is given by (\ref{eq:linear}) with $A=\begin{pmatrix} 0 & -1\\ 1 & -1\end{pmatrix}$. Solving the Lyapunov equation (\ref{eq:lyap}) with $Q=I$ gives $P=\begin{pmatrix} 1.5 & -0.5\\ -0.5 & 1\end{pmatrix}$. With $r=0.9999$, we can verify (\ref{eq:cr2}) with $c=0.29$. The ellipsoid $\Omega_c=\set{x\in\Real^n:\,x^TPx\le c}$ provides a verified local region of attraction for the origin.
\end{ex}

\subsection{Verification of Regions of Attraction in the Large}

Having verified a local ROA, we can use the learned neural Lyapunov function $W_{N}$ to enlarge the verified ROA. To do this, we can verifiy the following conditions using SMT:
\begin{align}
c_1\le {W}_N(x) \le c_2 &\Longrightarrow \dot{W}_N \le -\eps,\label{eq:dW}\\
{W}_N(x)\le c_1  &\Longrightarrow x^TPx\le c,\label{eq:WP}
\end{align}
where $\eps>0$, $c_2>c_1>0$, and $c$ and $P$ are from Section \ref{sec:local}. We can immediately show the following result. 

\begin{prop}
    If (\ref{eq:dW}) and (\ref{eq:WP}) hold for all $x\in X$ and the set $\W_{c_2}=\set{x\in X:\, W_N(x)\le c_2}$ does not intersect with the boundary of $X$, then $\W_{c_2}$ is contained in the ROA of the origin. 
\end{prop}

\begin{proof}
A solution starting from $\W_{c_2}$ remains in $\W_{c_2}$ as long as the solution do not leave $X$. However, to leave $X$ it has to cross the boundary of $\W_{c_2}$ first. This is impossible because of (\ref{eq:dW}). Within $\W_{c_2}$, solutions converge to $\Omega_c=\set{x\in\Real^n:\,x^TPx\le c}$ in finite time, which is a verified region of attraction.  
\end{proof}

\section{Numerical Examples}

In this section, we provide numerical examples to demonstrate the proposed approach for learning and verifying neural Lyapunov functions. For consistency, the reported computation times were recorded on a MacBook Pro equipped with a 2GHz Quad-Core Intel Core i5 processor and 16GB of memory. However, computations testing the feasibility of the work were performed on clusters provided by the Digital Research Alliance of Canada.

\begin{ex}[Reversed Van der Pol, continued]\label{ex:van_der_pol}
Revisit the reversed Van der Pol oscillator in Example \ref{ex:van_der_pol_local}. We know the real domain of attraction of the origin is bounded by a period orbit. To compute a neural Lyapunov function that can approximately capture the true ROA, we use a data-driven mechanism to generate data for the integral 
$$
V(x) = \int_0^\infty \norm{\phi(t,x)}^2dt
$$
as described in Section \ref{sec:algorithm}. A total of $90,000$ uniformly spaced samples on the domain $X=[-2.5,2.5]\times [-3.5,3.5]$ are generated. The total time for data generation is 1,793 seconds. The set $X$ contains the domain of attraction as a subset\footnote{The domain of attraction can be estimated roughly by numerical simulations or observational data. If there is no prior knowledge of $\D(\A)$, we can still arbitrarily specify a set $X$ containing $\A$, assuming we know $\A$. Then we run into a similar situation as in an unbounded domain of attraction. As shown in Example \ref{ex:poly}, although we cannot hope to approximate the entire domain of attraction, we can still obtain verified regions of attraction inside $X$.}. The terminating condition for numerical integration to determine $V$ is either $\norm{\phi(T,x)}\le 10^{-3}$ or $ \int_0^T \norm{\phi(t,x)}^2dt\ge 200$. We then set $W(x)=\tanh(\alpha V(x))$ with $\alpha=0.1$. The parameters chosen are similar to those in \cite{kang2021data}, except that we choose a larger set of samples to demonstrate the limit of the data-driven approach compared to a physics-informed neural network approach that incorporates Zubov's PDE in the training process. The ellipsoidal local region of attraction verified in Example \ref{ex:van_der_pol_local} is shown with the dashed line. 

We trained neural networks of different sizes as detailed in Table \ref{tab:nn_data}. A batch size of 32 was used, and the learning rate was chosen to be $10^{-3}$ after some testing. We terminated training when the mean-square training loss was smaller than $10^{-5}$ or after 200 epochs. The learned neural Lyapunov functions were verified using dReal with its Python binding \cite{gao2013dreal}. The computation times for training and verification are recorded in Table \ref{tab:nn_data}. Figure \ref{fig:ver_der_pol} depicts the result of training and verification with a learned Lyapunov function and a verified region of attraction close to the boundary of the domain of attraction.

In Table \ref{tab:nn_data}, a ``data-driven'' approach refers to a network trained with 90,000 samples without using Zubov's equation explicitly, even though as shown in \ref{sec:construct}, the construction results in a solution to Zubov's equation. This is close to the method proposed in \cite{kang2021data}. A PINN approach refers to learning a neural Lyapunov function only through evaluations of the residual errors at 90,000 training points. No data is required except for the evaluation of loss at these points. A ``PINN + data'' approach refers to a neural network trained with PINN and Zubov's PDE, but also incorporated a data loss in the training process. The number of data points chosen is 900, which is 1\% of the total number of samples in a data-driven approach. It is evident from the results in this table that a PINN+data approach performs the best in terms of obtaining a larger certified region of attraction. 
\end{ex}

\begin{figure}[htbp]
  \centering
  \includegraphics[width=0.49\textwidth]{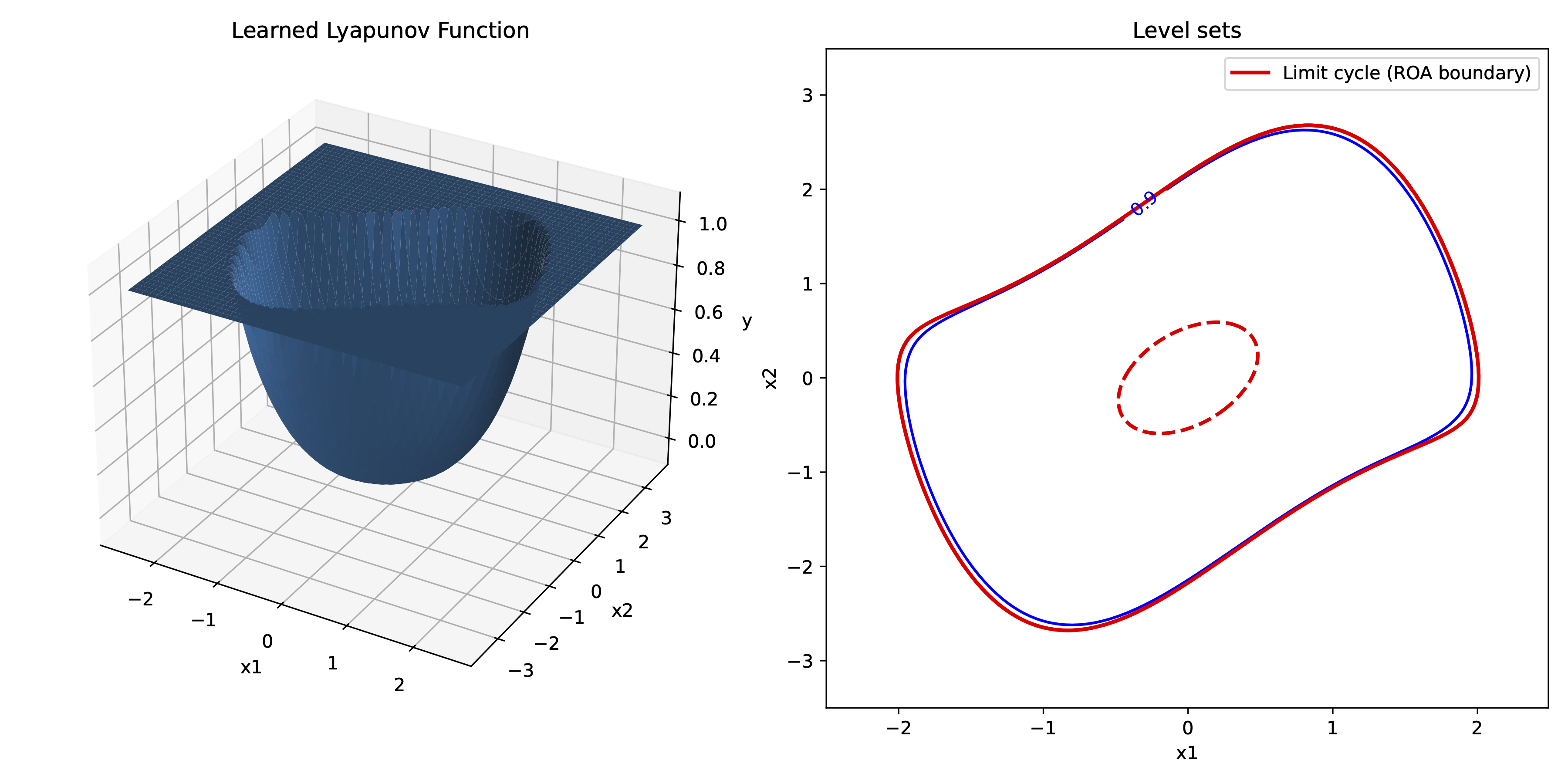}
  \caption{A neural network, consisting of three hidden layers with 10 neurons per layer and trained using Zubov's PDE (\ref{DV:zubov}), is capable of approximately representing a verifiable Lyapunov function that closely approximates the domain of attraction for the stable equilibrium point of the Van der Pol equation. The dashed lines represent a local verified region of attraction computed in Example \ref{ex:van_der_pol_local}. }
  \label{fig:ver_der_pol}
\end{figure}

\begin{table*}[h]
\caption{Verification of Neural Lyapunov Functions for Van der Pol Equation (Example \ref{ex:van_der_pol})}
\centering
\begin{tabular}{|c|c|c|c|c|c|c|c|c|c|c|}
\hline
\textbf{Approaches} & \textbf{Layer} & \textbf{Width} & \makecell{\textbf{No. of}\\ \textbf{params.}} & \makecell{\textbf{Data gen.}\\ \textbf{time}} & \makecell{\textbf{Training}\\ \textbf{time}} & \textbf{Epochs} & \textbf{Final loss} &  \makecell{\textbf{Verification}\\ \textbf{time}} & \makecell{\textbf{Verified level}\\ ($W_N(x)\le c_2$)} & \makecell{\textbf{Volume}\\ \textbf{\%}} \\
\hline
Data-driven & 2 & 10 & 151 & 2,156 (s) & 326 (s) & 200 & $9\times 10^{-5}$ & 46 (s) & 0.7 & 84.93\%\\
\hline
PINN & 2 & 10 & 151 & 0 & 636 (s) & 200 & $11 \times 10^{-5}$ & 24 (s) & 0.84 & 87.20\% \\
\hline
PINN + data (1\%) & 2 & 10 & 151 & 216 (s) & 943 (s) & 200 & $56 \times 10^{-5}$ & 31 (s) & 0.77 & \textbf{87.45\%}\\
\hline
\hline
Data-driven & 2 & 30 & 1,051 & 2,156 (s) & 363 (s) & 200 & $3\times 10^{-5}$ & 3,207  & 0.82 & 92.52\%\\ 
\hline
PINN & 2 & 30 & 1051 & 0 & 729 (s) & 200 & $36 \times 10^{-5}$ & 538 (s) & 0.86 & 91.55\%\\
\hline
PINN + data (1\%) & 2 & 30 & 1051 & 216 (s) & 1,293 (s) & 200 & $15 \times 10^{-5}$ &  842 (s) & 0.86 & \textbf{94.22\%}
\\
\hline
\hline
Data-driven & 3 & 10 & 261 & 2,156 (s) & 381 (s) & 200 & $10^{-5}$ & 51,301 (s) & 0.86 & 95.00\%\\
\hline
PINN & 3 & 10 & 261 & 0 & 129 (s)  & 32 & $9 \times 10^{-6}$ &  8,416 (s) & 0.88 & 93.42\% \\
\hline
PINN + data (1\%) & 3 & 10 & 261 & 216 (s) & 593 (s) & 107 & $8 \times 10^{-6}$ &  24,533 (s) & 0.90 & \textbf{96.31\%}\\
\hline
\end{tabular}
\label{tab:nn_data}
\end{table*}

\begin{ex}[Polynomial system]\label{ex:poly}
We consider the polynomial system in \cite{mauroy2016global}:
\begin{align*}
\dot x_1 &= x_2, \\
\dot x_2 &= -2x_1 + \frac13 x_1^3 - x_2.
\end{align*}
\end{ex}
The origin of this system is known to have an unbounded domain of attraction, delimited by the stable manifolds of the saddle equilibrium points at $(\pm\sqrt{6},0)$. We restricted our computations to the domain $[-6,6]\times [-6,6]$ and similarly generated 90,000 samples. The computational results for training and verification are summarized in Table \ref{tab:nn_data_poly}. Figure \ref{fig:poly} depicts the learned Lyapunov function, and the verified region of attraction using a neural network with three hidden layers of width 10. It can be seen that even though the actual domain of attraction is unbounded, the proposed algorithm approximates the largest region of attraction that can be certified by the choice of neural Lyapunov function. Similar to Example \ref{ex:van_der_pol_local}, we were able to verify a local region of attraction of the form
$\Omega_c=\set{x\in\Real^n:\,x^TPx\le c}$, 
using the SMT solver dReal \cite{gao2013dreal} by verifying (\ref{eq:rQ}) and (\ref{eq:cr2}) with $P=\begin{pmatrix}1.75 & 0.25 \\ 0.25 & 0.75\end{pmatrix}$, $Q=I$, $r=0.9999$, and $c=2.4$. The region $\Omega_c$ is depicted with a dashed line in Figure \ref{fig:poly}. 

\begin{figure}[htbp]
  \centering
\includegraphics[width=0.49\textwidth]{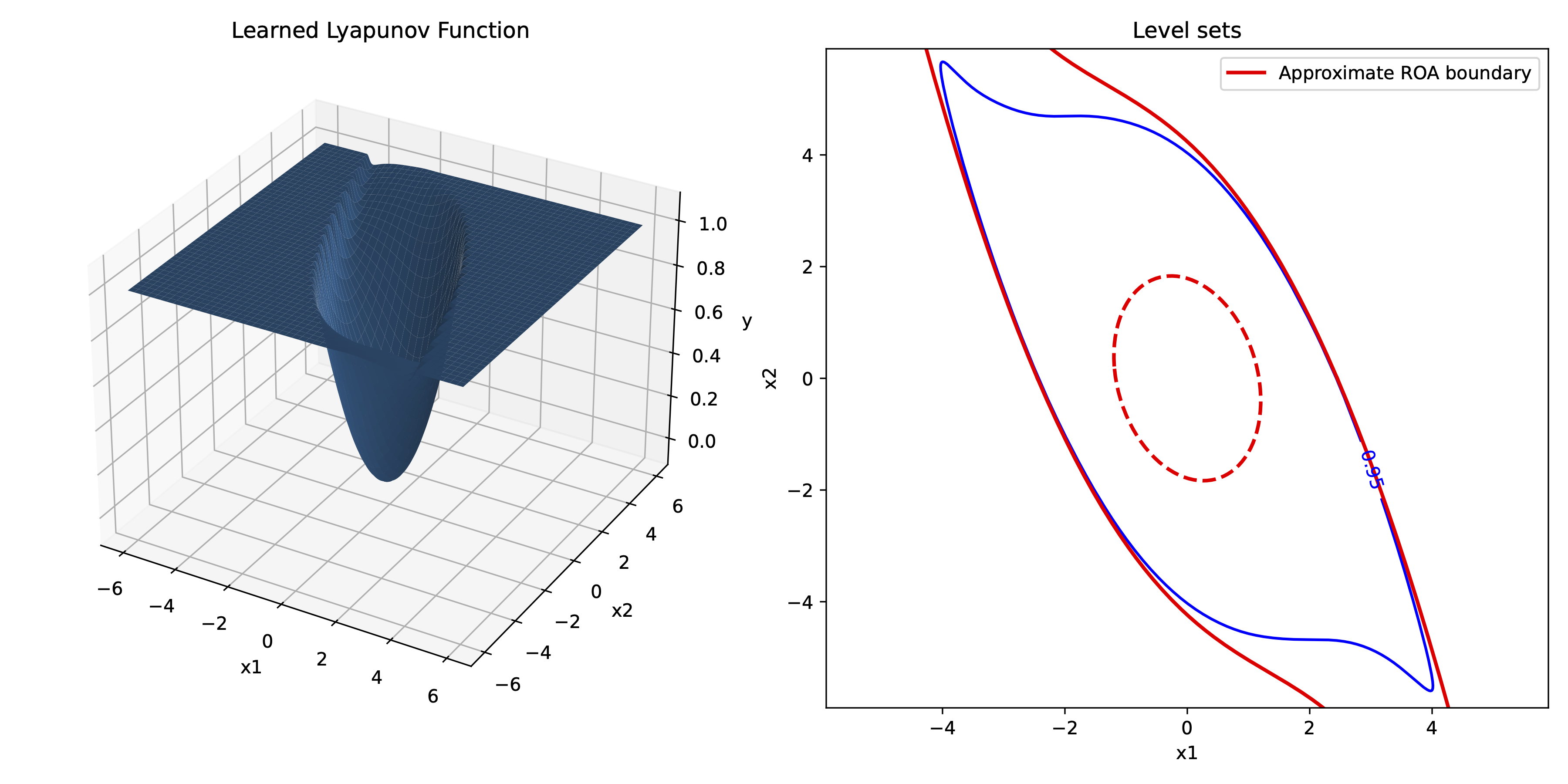}
  \caption{A neural network, consisting of three hidden layers with 10 neurons per layer and trained using Zubov's PDE (\ref{DV:zubov}), is capable of  representing a verifiable Lyapunov function that captures a region of attraction for the stable equilibrium point of the polynomial system in Example \ref{ex:poly} inside the domain $[-6,6]\times [-6,6]$, even though the actual domain of attraction for the origin is unbounded. The dashed lines represent a local verified region of attraction. }
  \label{fig:poly}
\end{figure}

\begin{table*}[h]
\caption{Verification of Neural Lyapunov Functions for Example \ref{ex:poly}}
\centering
\begin{tabular}{|c|c|c|c|c|c|c|c|c|c|c|}
\hline
\textbf{Approaches} & \textbf{Layer} & \textbf{Width} & \makecell{\textbf{No. of}\\ \textbf{params.}} & \makecell{\textbf{Data gen.}\\ \textbf{time}} & \makecell{\textbf{Training}\\ \textbf{time}} & \textbf{Epochs} & \textbf{Final loss} &  \makecell{\textbf{Verification}\\ \textbf{time}} & \makecell{\textbf{Verified level}\\ ($W_N(x)\le c_2$)} & \makecell{\textbf{Volume}\\ \textbf{\%}} \\
\hline
PINN + data & 2 & 10 & 151 & 97 (s) & 989 (s) & 200 & $37 \times 10^{-5}$ &  16 (s) & 0.91 & {74.15\%}\\
\hline
PINN + data & 2 & 30 & 1,051 & 97 (s) & 1,075 (s) & 200  & $2 \times 10^{-5}$ &  295 (s) & 0.95 & {\textbf{85.03}\%}\\
\hline
PINN + data & 3 & 10 & 261 & 97 (s) & 725 (s) & 123 & $6 \times 10^{-6}$ &  2,990 (s) & 0.95 & {84.92\%}\\
\hline
\end{tabular}
\label{tab:nn_data_poly}
\end{table*}

\section{Conclusions}

In this paper, we present a physics-informed neural network approach for learning nearly maximal Lyapunov functions using Zubov's equation. Our work provides theoretical justification on the existence of such Lyapunov functions that can be approximated by neural networks and verified using SMT solvers. The proposed algorithms allow for the computation of Lyapunov functions that can approximate the entire domain of attraction for an asymptotically stable compact set. We also provide formal verification of the Lyapunov candidate generated by the neural network, both locally and globally, which allows us to find verified regions of attraction that are close to the domain of attraction. Our work demonstrates that combining a physics-informed neural network approach with a small number of data points can improve the results of learning and verification.

One potential limitation of the work is its scalability to high-dimensional systems, as is the case with other computational approaches to constructing Lyapunov functions. Future work could focus on investigating the scalability of the proposed approach. It is believed that a compositional structure can help break the curse of dimensionality in learning function approximations using neural networks \cite{poggio2017and}. Both \cite{kang2021data} and \cite{grune2021computing,grune2021overcoming} have discussed this in the context of approximating Lyapunov functions. It would be interesting to investigate whether such compositionality can alleviate the curse of dimensionality in the formal verification of Lyapunov functions as well. Another interesting direction is to provide theoretical guarantees of learning a solution to Zubov's equation if one exists. Combining computation and verification of neural Lyapunov functions with a policy iteration scheme to provide verifiable guarantees of neural policy iteration is another promising research direction.

\bibliographystyle{plain}
\bibliography{cdc23}

\end{document}